\documentclass[12pt]{amsart}

\usepackage{amssymb}
\usepackage{mathtools}

\usepackage{graphicx}
\usepackage{color}

\newtheorem{theorem}{Theorem}[section]

\newtheorem{corollary}[theorem]{Corollary}

\newtheorem{lemma}[theorem]{Lemma}
\newtheorem{proposition}[theorem]{Proposition}

\theoremstyle{definition}
\newtheorem{definition}[theorem]{Definition}

\theoremstyle{remark}

\numberwithin{equation}{section}

\newcommand{\adj}{\operatorname{adj}}
\newcommand{\diag}{\operatorname{diag}}
\newcommand{\li}[1]{\mathcal{#1}}
\newcommand{\mb}[1]{\mathbf{#1}}
\newcommand{\trace}{\operatorname{tr}}
\newcommand{\wt}{\operatorname{wt}}
\newcommand{\D}{\mathrm{d}}
\DeclarePairedDelimiter\p{\lparen}{\rparen}

\begin{document}

\title{A combinatorial proof of a formula of Biane and Chapuy}

\author{Sinho Chewi}
\email{chewisinho@berkeley.edu}
\address{EECS Department, University of California, Berkeley, CA 94720, U.S.A.}

\author{Venkat Anantharam}
\email{ananth@berkeley.edu}
\address{271 Cory Hall, EECS Department, University of California, Berkeley, CA 94720, U.S.A.}
\thanks{Research supported by the NSF Science and Technology Center grant CCF-0939370:
``Science of Information", the NSF grants 
ECCS-1343398, CNS-1527846 and CIF-1618145, and by the William and Flora Hewlett
Foundation supported Center for Long Term Cybersecurity at Berkeley.}


\date{\today}



\begin{abstract}
    Let $G$ be a simple strongly connected weighted directed graph.
Let $\mathcal{G}$ denote the spanning tree graph of $G$. That is, the vertices of $\mathcal{G}$ consist of the directed rooted spanning trees on $G$, and the edges of $\mathcal{G}$ consist of pairs of trees $(t_i, t_j)$ such that $t_j$ can be obtained from $t_i$ by adding the edge from the root of $t_i$ to the root of $t_j$ and deleting the outgoing edge from the root of $t_j$. 
A formula for the ratio of the sum of the weights of the directed rooted spanning trees on $\mathcal{G}$ to the sum of the weights of the directed rooted spanning trees on $G$ was recently given by Biane and Chapuy. We provide an alternative proof of this formula, which is both simple and combinatorial. The proof involves working with the stochastic zeta function of an irreducible Markov chain.
By generalizing the stochastic zeta function we also recover the general result of Biane and Chapuy which gives a formula for the determinant of the Schr\"{o}dinger matrix on $\mathcal{G}$ corresponding to a given Schr\"{o}dinger matrix on $G$, in terms of the minors of the latter matrix.
\end{abstract}

\maketitle


.

\section{Introduction}

Let $G := (V, E)$ be a simple strongly connected weighted directed graph. Let $x_e$ denote the weight of edge $e$, which we think of as a variable.
$\li{G} := (\li{V}, \li{E})$ denotes the
spanning tree graph of $G$. The vertex set $\li{V}$ is the set of directed rooted spanning trees on $G$. If $t_i, t_j \in \li{V}$ are such that $i$ is the root of $t_i$, $j$ is the root of $t_j$, and $t_j$ can be obtained from $t_i$ by adding the edge $(i, j)$ and removing the outgoing edge from $j$, then the edge $(t_i, t_j)$ belongs to $\li{E}$. We assign this edge in $\li{E}$ the weight $x_{(i, j)}$. Thus $\li{G}$ is also a simple strongly connected weighted directed graph (see \cite[Lemma 9]{propp1998get} for a proof that $\li{G}$ is strongly connected).

In the paragraph above and in the rest of the paper we use $:=$ for equality by definition. 
The term \emph{graph} will always be used to mean a simple weighted directed graph.
Further, the term \emph{spanning tree} (resp.\ \emph{spanning forest}) will always be used to mean a directed rooted spanning tree (resp.\ directed rooted spanning forest).

Define the \emph{weight} of a forest to be the product of the weights of the edges included in the forest.
Recently, Biane and Chapuy \cite{biane2015laplacian}
consider
a rational function in the edge weights, denoted
$\Phi_G$, which can be interpreted as the ratio of the sum of the weights of spanning trees on $\li{G}$ to the sum of the weights of spanning trees on $G$. 
It is shown that $\Phi_G$ 
is a polynomial 
and a formula for it is given, see \cite[Theorem 3.6]{biane2015laplacian}.

The purpose of this paper is to provide a short combinatorial proof of this formula and of the more general formula of \cite[Theorem 3.5]{biane2015laplacian} which expresses in terms of the minors of a Schr\"{o}dinger matrix on $G$ the determinant of the corresponding Schr\"{o}dinger matrix on $\li{G}$. We defer the introduction of Schr\"{o}dinger matrices to Section \ref{scn:schrodinger}.

Our proof is developed in the framework of irreducible discrete-time Markov chains
on the vertex set $V$ whose positive transition probabilities,
except for self transitions,
correspond to the edges in $E$.
We call such a Markov chain \emph{associated} with $G = (V,E)$.
Our proof uses the connection between the sum of the weights of spanning trees on $G$, when the edge weights are given by the transition probabilities, and the stochastic zeta function of the Markov chain as defined in \cite{lind2001spanning,parrywilliams1977}.
Given an irreducible discrete-time Markov chain with transition probability matrix $P := (p(i, j); i, j \in V)$, where, for $i \neq j$,
we have $p(i,j) > 0$ iff $(i,j) \in E$,
the Laplacian matrix, $L := I - P$, will play a role in the 
discussion. $L$ is a $|V| \times |V|$ matrix with entries
\begin{align*}
    L_{i,j} = \begin{cases}
        -p(i, j), & \mbox{ if $(i, j) \in E$,} \\
        \sum_{k \ne i} p(i, k), & \mbox{ if $i = j$,} \\
        0, & \text{otherwise}.
    \end{cases}
\end{align*}

We begin in Section \ref{scn:cycle} with background information about Biane's polynomial $\Phi_G$ (initially thought of as a rational function) and the stochastic zeta function of an irreducible finite state discrete-time Markov chain. In Section \ref{scn:lift-cycles}, we relate the cycle structure of the spanning tree graph $\li{G}$ to the cycle structure of the original graph $G$. 
In Section \ref{scn:main}, we state and prove our main result, Theorem \ref{thm:main}, which gives an expression for the ratio of the zeta function of $\li{G}$ to that of $G$. 
We then prove Lemma \ref{l:m}, which gives a combinatorial interpretation to the $m(W)$ quantities in \cite{biane2015laplacian}.
The formula for the ratio of the sum of the weights of spanning trees on $\li{G}$ to the sum of the weights of the spanning trees on $G$ that appears in \cite[Theorem 3.6]{biane2015laplacian} then falls out as
Corollary \ref{cor:biane}.

In Section \ref{scn:schrodinger}, we generalize the stochastic zeta function of \cite{lind2001spanning, parrywilliams1977} to a multivariable analogue aimed at accommodating vertex weights. This allows us to give a simple combinatorial proof of the result of \cite[Theorem 3.5]{biane2015laplacian}.

\section{Cycle Structure of the Spanning Tree Graph} \label{scn:cycle}

\subsection{Jacobi's derivative formula.}

We recall Jacobi's formula for the derivative of the determinant of a one-parameter family of matrices \cite{may1965}: if $A(s)$ is a $n \times n$ differentiable matrix, then
\begin{align} \label{eq:jacobi}
    \frac{\D}{\D s} \det A(s) = \trace\p*{ (\adj A(s)) \frac{\D}{\D s} A(s) }.
\end{align}
Here $\adj A(s)$ denotes the adjugate of $A(s)$, see \cite{hornjohnson} for the definition.

\subsection{The polynomial $\Phi_G$.}

Let $P$ be an irreducible transition probability matrix on $V$ whose associated graph is $G$.
Let $\li{P}$ denote the lift of the transition matrix $P$. That is, if $t_i \in \li{V}$ and $t_j \in \li{V}$ have roots $i \in V$ and $j \in V$ respectively and $(t_i, t_j) \in \li{E}$, then we have $\li{P}(t_i, t_j) = p(i,j)$, while we have $\li{P}(t_i, t_i) = p(i,i)$. Note that $\li{P}$ is associated with $\li{G}$. Let $\li{L}$ denote the Laplacian matrix of $\li{P}$.
It is known that the weights $w(t)$, as $t$ ranges over spanning trees $t \in \li{V}$, give an invariant measure for the $\li{P}$-chain, see e.g.\ \cite{anantharam1989proof} for a proof. Further, $\det \li{L}^{(t)}$ as $t$ ranges over spanning trees $t \in \li{V}$, where $\li{L}^{(t)}$ denotes the matrix $\li{L}$ with the row and column corresponding to $t$ removed, is also invariant for $\li{P}$. 
This can be seen as a consequence of Kirchoff's matrix tree theorem, one of whose consequences is that for any irreducible Markov chain the minor of the Laplacian matrix got by erasing the row and column corresponding to a given state is the sum of the weights of the spanning trees rooted at that state,
see e.g.\ \cite{chaikenkleitman1978} for a proof, and the Markov chain tree theorem as proved e.g.\ in \cite{anantharam1989proof}, which states that the stationary distribution is proportional to the sum of such weights. These observations motivate Biane \cite{biane2015polynomials} to define a rational function in the edge weights, denoted $\Phi_G$, by
\begin{align} \label{eq:biane-1}
    \det \li{L}^{(t)} = w(t) \Phi_G.
\end{align}
In fact, it is proved in \cite{biane2015polynomials} that
$\Phi_G$ is a polynomial in the edge weights. This will also be proved as part of our discussion.

By summing over $t \in \li{V}$, we obtain another representation for $\Phi_G$:
\begin{align*}
    \tau(\li{G}) = \tau(G) \Phi_G,
\end{align*}
where $\tau(G)$ and $\tau(\li{G})$ denote the sum of the weights of the spanning trees on $G$ and $\li{G}$ respectively. This is because the sum of $\det \li{L}^{(t)}$ as $t$ ranges over $\li{V}$ equals $\tau(\li{G})$ as a consequence of Kirchoff's matrix tree theorem, as discussed above.

\subsection{Stochastic zeta function.}

Let $P$ be the transition probability matrix of an irreducible discrete-time Markov chain on $V$. The stochastic zeta function associated to $P$ is defined by:
\begin{align} \label{eq:zeta}
    \zeta_G(s) := \exp\p*{ \sum_{n=1}^\infty \frac{s^n}{n} \sum_{c \in C_n} \wt(c) } = \frac{1}{\det(I - sP)},
\end{align}
where $C_n$ is the set of cycles of length $n$ and the \emph{weight} of a cycle $c = v_1 \dotsm v_n$ is $\wt(c) := p(v_1, v_2) \dotsm p(v_n, v_1)$. See \cite{lind2001spanning} for more details. We reserve the notation $w(\cdot)$ for the weights of spanning forests and $\wt(\cdot)$ for the weights of cycles to avoid confusion.

The relevance for us of the stochastic zeta function comes from
\cite[pg.\ 491]{lind2001spanning}, where it is shown that
$\tau(G)$ can be expressed 
as follows: 
\[
(1/\zeta_G)'(1) = - \tau(G).
\]
Here the derivative on the left is in $s$ and can be thought of either as a formal derivative or as the derivative from the left, since $\zeta_G(s)$ is well defined for all $s$ with absolute value strictly less than $1$.

We are therefore led to consider the ratio between the zeta function of the graph $G$ and the zeta function of the lifted graph $\li{G}$, namely:
\begin{align} \label{eq:ratio}
    R(s) := (\zeta_G/\zeta_{\li{G}})(s).
\end{align}
By differentiating the equation $(1/\zeta_{\li{G}})(s) = R(s)(1/\zeta_G)(s)$ 
we obtain:
\begin{align*}
    (1/\zeta_{\li{G}})'(s) = R'(s)(1/\zeta_G)(s) + R(s)(1/\zeta_G)'(s).
\end{align*}
Setting $s = 1$, we find $-\tau(\li{G}) = - R(1) \tau(G)$,
since $(1/\zeta_G)(1) = \det(I - P) = 0$. Therefore,
\begin{align}	\label{eq:PhiGFormula}
    \Phi_G = \frac{\tau(\li{G})}{\tau(G)} = R(1).
\end{align}

\subsection*{Remark.}

Although we will not need this fact for the proof, we note that $R(s)$ is a polynomial in $s$, which follows from the discussion in Section \ref{scn:lift-cycles} below, and we record an observation about the coefficient of the linear term of $R$ which may be of independent interest.

\begin{proposition}
    The coefficient of $s$ in $R(s)$ is the sum of the weights of the self-loops of $G$, minus the sum of the weights of the self-loops of $\li{G}$.
\end{proposition}
\begin{proof}
    Using Jacobi's matrix identity \eqref{eq:jacobi},
    \begin{align*}
        \frac{\D}{\D s} R(s)
        = \frac{-\trace(\adj(I - s \li{P}) \li{P})}{\det(I - s P)} + \frac{\det(I - s \li{P})}{\det(I - s P)^2} \trace(\adj(I - s P) P).
    \end{align*}
    Setting $s = 0$, we see that the coefficient of $s$ in $R(s)$ is $-(\trace \li{P} - \trace P)$. The claim follows.
\end{proof}

We leave open the question of finding other combinatorial interpretations of $R(s)$.

\subsection{Lifting cycles to the spanning tree graph.} \label{scn:lift-cycles}

The following result about the structure of the spanning tree graph was observed for the case of simple cycles, i.e.\ cycles without repeated vertices, in \cite[Section 2.6]{biane2015laplacian}. We give a new proof of the result which holds for general cycles.

\begin{proposition} \label{prop:lift}
    Each cycle in $\li{G}$ projects down onto a cycle of $G$, given by projecting each spanning tree to its root. Conversely, the number of cycles in $\li{G}$ that project to the cycle $c$ in $G$ is given by the number of spanning forests rooted in the vertices of the cycle.
\end{proposition}
\begin{proof}
    The first statement, that cycles in $\li{G}$ project down to cycles in $G$, is clear. For the second statement, let $c$ be a cycle in $G$, let $W$ denote the set of vertices of $c$, and fix a spanning forest $f$ rooted in $W$. Pick an arbitrary starting vertex $w \in W$. We claim that for any spanning tree $t \in \li{V}$ rooted at $w$ such that $f \subseteq t$, the path in $\li{G}$ starting from $t$ obtained by traversing each edge in $c$ ends at the same spanning tree $t_f$. Indeed, the edges which are in $f$ are unchanged throughout the path, and the outgoing edge of each $v \in W \setminus \{w\}$ at the end of the path is uniquely determined by the last transition out of $v$ in $c$. We have exhibited a bijection $f \mapsto t_f$, where $t_f$ is the unique starting vertex in $\li{G}$ containing $f$ for which the cycle $c$ in $G$ lifts to a cycle in $\li{G}$, which suffices to prove the claim.
\end{proof}

\section{Main Result} \label{scn:main}

We are ready to state and prove our main result.

\begin{theorem} \label{thm:main}
    $R(s)$, defined in \eqref{eq:ratio}, satisfies
    \begin{align*}
        R(s) = \prod_{\text{strongly connected } W \subset V} \det((I - sP)^{(V \setminus W)})^{m'(W)},
    \end{align*}
    where the product ranges over all proper strongly connected subsets of $V$. Here, $(I - sP)^{(V \setminus W)}$ denotes the matrix $I - sP$ with the rows and columns corresponding to $V \setminus W$ removed, and $m'(W)$ is defined recursively as
    \begin{align} \label{eq:m}
        m'(W) :=
        \begin{cases}
            1, & W = V \\
            k(W) - \sum_{W' \supset W} m'(W'), & W \ne V
        \end{cases}
    \end{align}
    where $k(W)$ is the number of spanning forests rooted in $W$.
\end{theorem}

For the rest of the section, $W$ and $W'$ will be used exclusively to mean a strongly connected subset of $V$, that is, all products and summations involving $W$ or $W'$ should be interpreted as ranging over strongly connected subsets only.

\begin{proof}[Proof of Theorem \ref{thm:main}]
    Let $C_n$ and $\li{C}_n$ denote the set of cycles of length $n$ in $G$ and $\li{G}$ respectively. Similarly, let $C_W$ (resp.\ $\li{C}_W$) denote the set of cycles $c$ in $G$ (resp.\ $\li{G}$) such that the set of vertices in $c$ (resp.\ the set of roots of the vertices in $c$) equals $W$. From the representation \eqref{eq:zeta},
    \begin{align*}
        R(s) &= \exp\left\{ \sum_{n=1}^\infty \frac{s^n}{n} \left( \sum_{c \in C_n} \wt(c) - \sum_{c \in \li{C}_n} \wt(c) \right)\right\} \\
             &= \exp\left\{ \sum_{W \subseteq V, W \ne \varnothing} \left(\sum_{c \in C_W} \frac{\wt(c) s^{|c|}}{|c|} - \sum_{c \in \li{C}_W} \frac{\wt(c) s^{|c|}}{|c|} \right)\right\},
    \end{align*}
    where $|c|$ denotes the length of $c$. From Proposition \ref{prop:lift}, each cycle $c$ that appears in the first summation appears exactly $k(W)$ times in the second summation. So,
    \begin{align} \label{eq:ratio-2}
        R(s) &= \exp\left\{ - \sum_{W \subseteq V, W \ne \varnothing} (k(W)-1) \sum_{c \in C_W} \frac{\wt(c) s^{|c|}}{|c|} \right\}.
    \end{align}

    We seek to replace the set $C_W$ in the summation with the set $\tilde{C}_W$, where $\tilde{C}_W$ is the set of cycles $c$ which only use vertices from $W$. The difference between $C_W$ and $\tilde{C}_W$ is that $\tilde{C}_W$ includes cycles which only use a proper subset of $W$, whereas $C_W$ only includes cycles which use every vertex in $W$ at least once. So, when we replace $C_W$ with $\tilde{C}_W$, we introduce over-counting. Indeed, for $c \in C_W$, we have $c \in \tilde{C}_{W'}$ for each $W' \supseteq W$. We claim:
    \begin{align} \label{eq:km}
        \sum_{\substack{W \subset V \\ W \ne \varnothing}} (k(W) - 1) \sum_{c \in C_W} \frac{\wt(c) s^{|c|}}{|c|} = \sum_{\substack{W \subset V \\ W \ne \varnothing}} m'(W) \sum_{c \in \tilde{C}_W} \frac{\wt(c) s^{|c|}}{|c|}.
    \end{align}
    Note that we restrict our attention to proper subsets of $V$, since $k(V) - 1 = 0$. To check that \eqref{eq:km} correctly adjusts for the over-counting, observe that $m'(W)$, as defined in \eqref{eq:m}, satisfies
    \begin{align} \label{eq:m-condition}
        k(W) - 1 = \sum_{\substack{W' \subset V \\ W \subseteq W'}} m'(W'), \qquad W \subset V, \; W \ne \varnothing.
    \end{align}
    We now have
    \begin{align*}
        R(s)
        &= \prod_{\substack{W \subset V \\ W \ne \varnothing}} \exp\left\{ - m'(W) \sum_{c \in \tilde{C}_W} \frac{\wt(c) s^{|c|}}{|c|} \right\} \\
        &= \prod_{\substack{W \subset V \\ W \ne \varnothing}} \exp\left\{ \sum_{c \in \tilde{C}_W} \frac{\wt(c) s^{|c|}}{|c|} \right\}^{-m'(W)}.
    \end{align*}
    We can rewrite the last expression by using:
    \begin{align*}
        \exp\left\{ \sum_{c \in \tilde{C}_W} \frac{\wt(c) s^{|c|}}{|c|} \right\}
        &= \exp\left\{ \sum_{n=1}^\infty \frac{s^n}{n} \trace\p*{ (P^{(V \setminus W)})^n } \right\} \\
        &= \exp\left\{ \sum_{n=1}^\infty \frac{s^n}{n} \sum_{i=1}^{|W|} \lambda_i^n \right\} \\
        &= \exp\left\{ \sum_{i=1}^{|W|} \log \frac{1}{1 - s \lambda_i} \right\}
        = \prod_{i=1}^{|W|} \frac{1}{1 - s \lambda_i} \\
        &= \frac{1}{\det(I - s P^{(V \setminus W)})},
    \end{align*}
    where $\lambda_i$, $i = 1, \dotsc, |W|$, are the eigenvalues of $P^{(V \setminus W)}$. 
Now, in the product expression for $R(s)$ in the statement of the theorem, we can let $W = \varnothing$ if we wish, with the convention $\det(I - s P^{(V)}) = 1$. The result follows.
\end{proof}

Our next goal is to see how Theorem \ref{thm:main} implies the result of \cite[Theorem 3.6]{biane2015laplacian}.

In \cite{biane2015laplacian}, Biane and Chapuy describe an exploration algorithm, which we briefly summarize as follows: Fix a spanning tree $t \in \li{V}$ with root $w \in V$ and an arbitrary ordering of the vertices. Perform a modified breadth-first traversal of the graph (following incoming edges in order of increasing source vertex) such that if a vertex $v$ is first visited in the traversal via an edge that is not included in $t$, then delete $v$ and all edges associated with $v$ from the graph before continuing the traversal. Let $\psi(t)$ denote the strongly connected subset containing $w$ among the remaining vertices after the algorithm terminates. For a non-empty strongly connected subset $W \subseteq V$, fix some $w \in W$ and let $m(W, w)$ denote the number of spanning trees $t$ rooted at $w$ such that $\psi(t) = W$.

It is proved in \cite{biane2015laplacian} that the quantity $m(W, w)$ does not depend on the choice of $w$, nor on the choice of ordering of the vertices during the traversal. We give another proof of this fact by giving $m(W, w)$ a different combinatorial interpretation. Instead of working with an ordering on the vertices, we instead give the proof for the more general setting in which the exploration algorithm is carried out with an arbitrary ordering on the edges.

\begin{lemma} \label{l:m}
    Fix a non-empty strongly connected subset $W \subseteq V$ and a vertex $w \in W$. The number of spanning trees $t$ rooted at $w$ such that $W \subseteq \psi(t)$ equals the number of spanning forests rooted at $W$.
\end{lemma}
\begin{proof}
    We will prove that for each spanning forest $f$ rooted at $W$, there is exactly one spanning tree $t$ on $V$ rooted at $w$ such that $f \subseteq t$ and $\psi(t) \supseteq W$, from which the claim follows.

    Fix a spanning forest $f$. Generate a spanning tree $t_f$ on $V$ by running a breadth-first traversal (following incoming edges) starting from $w$ with the following rules:
    \begin{itemize}
        \item Initialize $t_f$ with the edges in $f$.
        \item If a vertex $v \in V \setminus W$ is first visited through an edge not included in $f$, then delete the vertex $v$ and all of its associated edges.
        \item If a vertex $v \in W$ is first visited through edge $e$, add $e$ to $t_f$.
    \end{itemize}
    By construction, running the exploration algorithm on $t_f$ yields $\psi(t_f) \supseteq W$. Conversely, suppose that $t$ is another spanning tree on $V$ rooted at $w$ with $f \subseteq t$, $t \ne t_f$. Let $v$ denote the first vertex in $W$ visited during the construction of $t_f$ such that the outgoing edge from $v$ is different in $t$ and $t_f$. Then, the exploration algorithm on $t$ will delete the vertex $v$ during the traversal, and so we have $W \not\subseteq \psi(t)$.
\end{proof}

\begin{proposition} \label{prop:m-m}
    For non-empty $W$, for all $w \in W$, $m(W, w) = m'(W)$. In particular, $m(W, w)$ does not depend on the choice of $w$, nor on the ordering of the edges.
\end{proposition}
\begin{proof}
Fix an ordering on the edges. We prove the claim by backwards induction on the size of $W$.

By definition, $m'(V) = 1$, and indeed $m(V, w) = 1$ for any $w \in V$ (see \cite[Lemma 3.3]{biane2015laplacian}). Now, suppose that the claim holds for all $W' \supset W$ for some $W$. Fix $w \in W$. From Lemma \ref{l:m}, $k(W) = \sum_{W' \supseteq W} m(W', w)$, so
\begin{align*}
    m(W, w) = k(W) - \sum_{W' \supset W} m(W', w) = k(W) - \sum_{W' \supset W} m'(W') = m'(W)
\end{align*}
by the inductive claim and the definition of $m'$. Also, $m'(W)$ does not depend on the choice of $w$ nor on the ordering on the edges, which establishes the result.
\end{proof}

In light of Proposition \ref{prop:m-m}, we write $m(W)$ for $m(W, w)$, as also done in \cite{biane2015laplacian}, and have established $m(W) = m'(W)$ for $W \ne \varnothing$. Consequently, we have:

\begin{corollary}[Theorem 3.6 in \cite{biane2015laplacian}] \label{cor:biane}
    The polynomial $\Phi_G$ satisfies
    \begin{align}	\label{eq:PhiGBiane}
        \Phi_G = \prod_{\text{strongly connected } W \subset V} (\Psi_{V \setminus W})^{m(W)},
    \end{align}
    where the product ranges over all proper strongly connected subsets of $V$ and $\Psi_{V \setminus W}$ is the sum of the weights of spanning forests rooted at $V \setminus W$.
\end{corollary}
\begin{proof}
    This is a direct consequence of Theorem \ref{thm:main} (set $s=1$), Proposition \ref{prop:m-m}, and the Kirchhoff-Chaiken-Chen matrix forest theorem, which states that $\det(I - P^{(W)})$ is the sum of the weights of spanning forests rooted at $W$, see e.g.\ \cite{chaikenkleitman1978} or \cite{pitman2016tree}.
\end{proof}

\subsection*{Remark.}

Once one recognizes that $R(s)$, as defined in \eqref{eq:ratio}, satisfies \eqref{eq:PhiGFormula} the proof of \eqref{eq:PhiGBiane} could have been carried out along the lines of \cite[Corollary 2.3]{athanasiadis1996}. Our proof however brings out the connections with the exploration algorithm of \cite{biane2015laplacian}. Further, as shown in Section \ref{scn:schrodinger} below, our technique leads to a simple combinatorial proof of the more general result of \cite[Theorem 3.5]{biane2015laplacian} concerning Schr\"{o}dinger matrices.

\section{Extension to Schr\"{o}dinger Matrices} \label{scn:schrodinger}

In addition to the set of variables $\{x_e : e \in E\}$ associated with the edges of the graph $G = (V,E)$, let $\{y_v : v \in V\}$ be a set of variables associated with the vertices of $G$. The corresponding \textit{Schr\"{o}dinger matrix} is the $|V| \times |V|$ matrix $H := Q + Y$, where
\begin{align*}
    Q(i, j) := \begin{cases}
        - x_{(i, j)}, & \mbox{ if $(i,j) \in E$,} \\
        \sum_{k : (i,k) \in E} x_{(i, k)}, & \mbox{ if $i = j$,}
    \end{cases} 
\end{align*}
and
\[
Y := \diag(y_1, \dotsc, y_{|V|}).
\]

In \cite{biane2015laplacian} Biane and Chapuy consider the Schr\"{o}dinger matrix on 
$\li{G} = (\li{V}, \li{E})$ corresponding to $H$ on $G = (V,E)$. To define this, associate the edge weight $x_{(i,j)}$ to $(t_i,t_j) \in \li{E}$, where $t_i$ has root $i$ and $t_j$ has root $j$, and associate the vertex weight $y_i$ to $t_i \in \li{V}$ when $t_i$ has root $i$.

The lifted counterpart of $H$ is then $\li{H} := \li{Q} + \li{Y}$, where:
    
    $\li{Q}(t_i, t_j) := -x_{(i, j)}$ if $t_i \ne t_j$ and $t_j$ is the spanning tree got from $t_i$ by adding the edge $(i, j)$ and removing the outgoing edge from $j$ (where $i$ and $j$ are the roots of $t_i$ and $t_j$ respectively); otherwise, $\li{Q}(t_i, t_i) := \sum_{t \ne t_i} \li{Q}(t_i, t)$;
    
     $\li{Y}$ is the diagonal matrix with $(t_i, t_i)$ entry equal to $y_i$, where $i$ is the root of $t_i$.\\
Note that we denote Schr\"{o}dinger matrices on $G$ (resp.\ $\li{G}$) by $H$ (resp.\ $\li{H}$) instead of $L$ (resp.\ $\li{L}$) as was done in \cite{biane2015laplacian}, to avoid confusion with the respective Laplacian matrices.

We now prove the analogue of Theorem \ref{thm:main} with Schr\"{o}dinger matrices. We would like to thank Biane and Chapuy for suggesting that we study this more general problem, after seeing an earlier version of this document.

\subsection{Generalization of the stochastic zeta function.}

Given the graph $G = (V, E)$, let $\mb{s} := (s_1, \dotsc, s_{|V|})$ denote a vector of vertex weights. For a cycle $c \in C_n$, with $c = v_1 \dotsm v_n$, its \textit{$\mb{s}$-weight} is defined to be $\wt(c; \mb{s}) := s_{v_1} \dotsm s_{v_n} p(v_1, v_2) \dotsm p(v_n, v_1)$.

\begin{definition}
    Given a Markov chain associated with the graph $G = (V,E)$ and having the transition probability matrix $P$, the \textit{vertex-weighted stochastic zeta function} associated with $P$ is defined as:
    \begin{align*}
        \zeta_G(\mb{s}) := \exp\p*{ \sum_{n=1}^\infty \frac{1}{n} \sum_{c \in C_n} \wt(c; \mb{s}) }.
    \end{align*}
\end{definition}

By setting $s_1 = \cdots = s_{|V|} = s$, we recover the single-variable stochastic zeta function \eqref{eq:zeta}. Next, we give the corresponding determinant expression for $\zeta_G(\mb{s})$.

\begin{theorem} \label{thm:zeta-multi}
    $\zeta_G(\mb{s})$ satisfies
    \begin{align} \label{eq:zeta-multi}
        \zeta_G(\mb{s}) = \frac{1}{\det(I - SP)},
    \end{align}
    where $S := \diag(s_1, \dotsc, s_{|V|})$.
\end{theorem}
\begin{proof}
   Let $C(u) := \det(I - u SP)$, where $u$ is a variable. From the Jacobi derivative formula of \eqref{eq:jacobi} we have
\[
C'(u) = - \trace\p*{ (\adj (I - uSP)) SP }.
\]
Hence,
\begin{align*}
-u \frac{C'(u)}{C(u)} &= \trace\p*{ \frac{1}{C(u)} (\adj (I - uSP)) uSP}\\
&= \trace\p*{ (I -uSP)^{-1} - I }.
\end{align*}
   Thus we get:
    \begin{align*}
        \sum_{n=1}^\infty \sum_{c \in C_n} \wt(c; \mb{s}) u^n = -u \frac{C'(u)}{C(u)}.
    \end{align*}
    Hence,
    \begin{align*}
        u \frac{\D}{\D u} \sum_{n=1}^\infty \frac{1}{n} \sum_{c \in C_n} \wt(c; \mb{s}) u^n = -u \frac{\D}{\D u} \ln C(u),
    \end{align*}
    and we deduce:
    \begin{align*}
        \exp\p*{ - \sum_{n=1}^\infty \frac{1}{n} \sum_{c \in C_n} \wt(c; \mb{s}) u^n } = \det(I - uSP)
    \end{align*}
    by integration, since both sides equal $1$ at $u = 0$. Now, set $u = 1$.
\end{proof}

\subsection{The general formula of Biane and Chapuy.}

Using the determinant expression \eqref{eq:zeta-multi}, we can now emulate the proof of Theorem \ref{thm:main}. As before, $W$ and $W'$ will be used exclusively to refer to strongly connected subsets of $V$.

\begin{theorem}
    Let $S$ be defined as in Theorem \ref{thm:zeta-multi} and $\li{S}$ denote its lifted counterpart, that is, $\li{S}$ is a diagonal matrix with $(t_i, t_i)$ entry equal to $s_i$ for each $t_i \in \li{V}$ that has root $i$. Then,
    \begin{align} \label{eq:sp-formula}
        \frac{\det(I - \li{SP})}{\det(I - SP)} = \prod_{\text{strongly connected } W \subset V} \det((I - SP)^{(V \setminus W)})^{m(W)},
    \end{align}
    where $m(W)$ is defined as in \eqref{eq:m} or in \cite{biane2015laplacian}.
\end{theorem}
\begin{proof}
    The proof of Theorem \ref{thm:main} goes through as before, with $\wt(c) s^{|c|}$ replaced with $\wt(c; \mb{s})$. In particular, the step
    \begin{align*}
        \exp\left\{ \sum_{c \in \tilde{C}_W} \frac{\wt(c; \mb{s})}{|c|} \right\} = \frac{1}{\det((I - SP)^{(V \setminus W)})}
    \end{align*}
    follows also from Theorem \ref{thm:zeta-multi}.
\end{proof}

By multiplying both sides of \eqref{eq:sp-formula} by $\det(I - SP)$, we may let the product on the right-hand side range over strongly connected subsets $W \subseteq V$. Now, by identifying $p(i, j) = Q(i, j)/(1 - y_i)$ for $i \ne j$, $p(i, i) = 1 - \sum_{k \ne i} p(i, k)$, and $s_i = 1 - y_i$, we obtain $I - SP = I - (I - Q - Y) = H$, whereby we obtain:

\begin{corollary}[Theorem 3.5 in \cite{biane2015laplacian}]
    For a Schr\"{o}dinger matrix $H$ and its lift $\li{H}$,
    \begin{align*}
        \det \li{H} = \prod_{\text{strongly connected } W \subseteq V} \det(H^{(V \setminus W)})^{m(W)},
    \end{align*}
    where $m(W)$ is defined as in \eqref{eq:m} or in \cite{biane2015laplacian}.
\end{corollary}

\bibliographystyle{amsplain}

\end{document}